\documentclass{rmmcart}

 \usepackage{hyperref} 
\usepackage{comment}
\usepackage{bbm}
\usepackage{cleveref}
\newtheorem{property}{Property}[section]
\usepackage[toc,page]{appendix}

\newcommand{\RN}[1]{%
  \textup{\uppercase\expandafter{\romannumeral#1}}%
  }

 \def\theequation{\arabic{section}.\arabic{equation}}

 \title[Characterizations of Fractional Sobolev Spaces]{Characterizations of Fractional Sobolev Spaces From the Perspective of Riemann-Liouville Operators}
\author[Y. Li]{Yulong Li}
\address{Department of Mathematics and Statistics, University of Wyoming, Laramie, Wyoming, USA}
\email{yli25@uwyo.edu}

\date{\today}
\keywords{Riemann-Liouville fractional operators, weak fractional derivative, Fourier transform, regularity, decomposition, mixed derivative.}
\subjclass{26A06, 46N20, 26A33}

\begin{document}

 \begin{abstract}
Fractional Sobolev spaces $\widehat{H}^s(\mathbb{R})$ have been playing important roles in analysis of many mathematical subjects. In this work, we re-consider fractional Sobolev spaces under the perspective of fractional operators and establish characterizations on the Fourier transform of functions of fractional Sobolev spaces, thereby giving another equivalent definition.
 \end{abstract}
 \maketitle

\section{Introduction}
$\widehat{H}^s(\mathbb{R})$ has been exhibiting important usefulness in the study of theory of classical integer-order PDEs; and it is well known from the standard textbooks that fractional Sobolev spaces $\widehat{H}^s(\mathbb{R})$ could be defined in several ways, namely via Fourier transform, Gagliardo norm or interpolation spaces.

Our previous work (\cite{V.Ging&Y.Li}, \cite{Y.Li}) have suggested that usual fractional Sobolev spaces have been behaving new features in analysis of fraction-order differential equations due to the simultaneous appearing of left, right and mixed Riemann-Liouville derivatives. In this work we continue to explore usual fractional Sobolev spaces under the perspective of fractional calculus theory; and the main result in this work is Theorem~\ref{MAINTHEOREM}, which characterizes the Fourier transform of elements of $\widehat{H}^s(\mathbb{R})$ and thus gives anther equivalent definition of $\widehat{H}^s(\mathbb{R})$.
The material is organized as follows:
\begin{itemize}
\item Section~\ref{section:Notation} introduces the notations and conventions adopted throughout the material.
\item Section~\ref{section1} introduces the preliminary knowledge on fractional R-L operators.
\item Section~\ref{section2} outlines some existing results on the characterization of $\widehat{H}^s(\mathbb{R})$, which are obtained in~\cite{V.Ging&Y.Li} and will provide the whole abstract setting for the work.
\item Section~\ref{section3} establishes the main results.
\end{itemize}
\section{Notations}\label{section:Notation}
\begin{itemize}
\item All functions considered in this material are default to be real valued  unless otherwise specified.
\item $(f,g)$ and $ \int_\mathbb{R}fg$ shall be used interchangeably. Also, we denote integration $\int_A f$ on set $A$ without pointing out the variable unless it is necessary to specify.
\item  If $f,g\in L^2(\mathbb{R})$, $f=g$ means $f =g, a.e.$, unless stated otherwise.
\item $C_0^\infty(\mathbb{R})$ denotes the space of all infinitely differentiable functions with compact support in $\mathbb{R}$.
\item  $\mathcal{F}(u)$ denotes the Fourier transform of $u$ with specific expression defined in Definition~\ref{def:FT}, $ \widehat{u}$ denotes the Plancherel transform of $u$ defined in Theorem~\ref{thm:PAR}, which is well known that $ \widehat{u}$ is an isometry map from $L^2(\mathbb{R})$ onto $L^2(\mathbb{R})$ and coincides with $\mathcal{F}(u)$ if $u\in L^1(\mathbb{R})\cap L^2(\mathbb{R})$.
\item $u^\vee$ denotes  the inverse of Plancherel transform, and $*$ denotes convolution.
\end{itemize} 
\section{Preliminary}\label{section1} 
\subsection{Fractional Riemann-Liouville Integrals and Properties}
\begin{definition} \label{def:RLI}
Let $u:\mathbb{R} \rightarrow \mathbb{R}$ and  $\sigma >0$. The left and right Riemann-Liouville fractional integrals of order $\sigma$ are, formally respectively, defined as
\begin{align}
\boldsymbol{D}^{-\sigma}u(x)&:= \dfrac{1}{\Gamma(\sigma)}\int_{-\infty}^{x}(x-s)^{\sigma -1}u(s) \, {\rm d}s, \label{eq:LRLI}\\
\boldsymbol{D}^{-\sigma * }u(x) &:= \dfrac{1}{\Gamma(\sigma)}\int_{x}^{\infty}(s-x)^{\sigma-1}u(s) \, {\rm d}s,  \label{eq:RRLI}
\end{align}
where $\Gamma(\sigma)$ is the usual Gamma function. 
\end{definition}
\begin{property}[\cite{MR1347689}, p. 96]\label{prp:IntegrationExchange}
Given $0<\sigma$,
\begin{equation}
( \phi, \boldsymbol{D}^{-\sigma}\psi)=( \boldsymbol{D}^{-\sigma*}\phi, \psi),
\end{equation}
for $\phi\in L^p(\mathbb{R})$, $\psi \in L^q(\mathbb{R})$, $p>1, q>1$, $1/p+1/q=1+\sigma$.
\end{property}
\begin{property}[\cite{MR1347689}, Corollary 2.1]\label{semigroup-new}
Let $\mu,\sigma >0$, and $w\in C_0^\infty(\mathbb{R})$, then
\begin{equation}
\boldsymbol{D}^{-\mu} \boldsymbol{D}^{-\sigma} w=\boldsymbol{D}^{-(\mu+\sigma)}w  \quad \text{and} \quad \boldsymbol{D}^{-\mu*} \boldsymbol{D}^{-\sigma*}w = \boldsymbol{D}^{-(\mu+\sigma)*}w.
\end{equation}
\end{property}
\begin{property}[\cite{MR1347689}, pp. 95, 96]\label{pro:Translation}
Let $\mu>0$. Given $h\in \mathbb{R}$, define the translation operator $\tau_h$ as 
$\tau_h u(x) = u(x-h)$. Also, given $\kappa>0$, define the dilation operator
$\Pi_\kappa$ as $\Pi_\kappa u(x) = u(\kappa x)$.
Under the assumption that $\boldsymbol{D}^{-\mu}u$ and $\boldsymbol{D}^{-\mu*}u$ are well-defined, the following is true:
\begin{equation}
\begin{aligned}
\tau_h(\boldsymbol{D}^{-\mu}u)=\boldsymbol{D}^{-\mu}(\tau_hu)&,\quad \tau_h(\boldsymbol{D}^{-\mu*}u)=\boldsymbol{D}^{-\mu*}(\tau_hu)\\
\Pi_\kappa(\boldsymbol{D}^{-\mu}u)=\kappa^\mu \boldsymbol{D}^{-\mu}(\Pi_\kappa u)&,\quad
\Pi_\kappa(\boldsymbol{D}^{-\mu*}u)=\kappa^\mu \boldsymbol{D}^{-\mu*}(\Pi_\kappa u).
\end{aligned}
\end{equation}
\end{property}
\subsection{Fractional Riemann-Liouville Derivatives and Properties}
\begin{definition}\label{def:RLD}
Let $u:\mathbb{R} \rightarrow \mathbb{R}$. Assume $\mu >0$, $n$ is the smallest integer greater than $\mu$ (i.e., $n-1 \leq \mu < n$), and $\sigma = n- \mu$. The left and right Riemann-Liouville fractional derivatives of order $\mu$ are, formally respectively, defined as
\begin{align}
\boldsymbol{D}^{\mu} u &:= \dfrac{1}{\Gamma (\sigma)}\dfrac{{\rm d}^n}{{\rm d}x^n}\int_{-\infty}^{x}(x-s)^{\sigma-1}u(s) \, {\rm d}s, \label{5} \\
\boldsymbol{D}^{\mu*}u &:= \dfrac{(-1)^n}{\Gamma(\sigma)}\dfrac{{\rm d}^n}{{\rm d} x^n}\int_{x}^{\infty}(s-x)^{\sigma-1}u(s) \, {\rm d} s.\label{6}
\end{align}
\end{definition}
\begin{property}[\cite{V.Ging&Y.Li}] \label{prop:Boundedness}
Let $0<\mu $ and  $u\in C_0^\infty(\mathbb{R})$, then $\boldsymbol{D}^\mu u, \boldsymbol{D}^{\mu*} u \in L^p(\mathbb{R})$ for any $1\leq p<\infty$.
\end{property}
\begin{property}[ \cite{MR1347689}, p. 137] \label{lem:FTFD}
 Let $\mu > 0, u \in C_0^\infty(\mathbb{R})$, then
 \begin{equation} \label{eq:FTFDo}
 \mathcal{F}(\boldsymbol{D}^{\mu}u) = (2\pi i \xi)^{\mu} \mathcal{F}(u) \text{ and }
 \mathcal{F}(\boldsymbol{D}^{\mu*} u) = (-2\pi i \xi)^{\mu} \mathcal{F}(u), \quad \xi \ne 0, 
\end{equation}
where $\mathcal{F}(\cdot)$ is the Fourier Transform  as defined in Definition~\ref{def:FT} and the complex power functions are understood as $(\mp i\xi)^{\sigma}=|\xi|^\sigma e^{\mp  \sigma \pi i \cdot \emph{sign} (\xi)/2}$.
\end{property}
\begin{property}[\cite{V.Ging&Y.Li}]\label{pro:TranslationDerivative} Consider $\tau_h$ and $\Pi_\kappa$ defined in Property~\ref{pro:Translation}. Let $\mu>0$, $n-1\leq \mu<n$, where $n$ is a positive integer, then
\begin{equation}
\begin{aligned}
\tau_h(\boldsymbol{D}^{\mu}u)=\boldsymbol{D}^{\mu}(\tau_hu)&,\quad \tau_h(\boldsymbol{D}^{\mu*}u)=\boldsymbol{D}^{\mu*}(\tau_hu)\\
\Pi_\kappa(\boldsymbol{D}^{\mu}u)=\kappa^{-\mu}\boldsymbol{D}^{\mu}(\Pi_\kappa u)&,\quad
\Pi_\kappa(\boldsymbol{D}^{\mu*}u)=\kappa^{-\mu}\boldsymbol{D}^{\mu*}(\Pi_\kappa u).
\end{aligned}
\end{equation}
\end{property}
\section{Characterization of Sobolev Space $\widehat{H}^s(\mathbb{R})$}\label{section2}
In this section, we will list the necessary concepts and results developed in ~\cite{V.Ging&Y.Li}, which characterize the classical Sobolev space $\widehat{H}^s(\mathbb{R})$ defined in~\ref{thm:FTHsR}. This section will give us the theoretical framework in which the main results in Section~\ref{section3} will be established.
%
%
%
%
%
%
%
%
 %
 \begin{definition}[Weak Fractional Derivatives\cite{V.Ging&Y.Li}]\label{def:WFD}
 Let $\mu>0$, and $u, w \in L^1_{loc}(\mathbb{R})$.  The function $w$ is called weak
 $\mu$-order left fractional derivative of $u$, written as  $\boldsymbol{D}^\mu u = w$,
 provided
 \begin{equation}
(u, \boldsymbol{D}^{\mu* } \psi) = (w, \psi)~\forall \psi \in C_0^\infty(\mathbb{R}).
 \end{equation}
In a similar faschion, $w$ is weak $\mu$-order right fractional derivative of $u$, written as 
$\boldsymbol{D}^{\mu*} u=w$,  provided
 \begin{equation}
 (u, \boldsymbol{D}^{\mu}\psi) = (w, \psi) ~\forall \psi \in C_0^\infty(\mathbb{R}).
 \end{equation}
 \end{definition}
   \begin{definition}[\cite{V.Ging&Y.Li}]\label{def:FractionalSobolevSpaces}
 Let $s\geq 0$. Define spaces
 \begin{equation}
  \widetilde{W}^{s}_L(\mathbb{R})=\{u\in L^2(\mathbb{R}), \boldsymbol{D}^s u \in L^2(\mathbb{R})\}, 
 \end{equation}
 \begin{equation}
  \widetilde{W}^{s}_R(\mathbb{R})=\{u\in L^2(\mathbb{R}), \boldsymbol{D}^{s*} u\in L^2(\mathbb{R})\}, 
 \end{equation}
where $\boldsymbol{D}^s u$ and $\boldsymbol{D}^{s*} u$ are in the weak fractional derivative sense as defined in Definition~\ref{def:WFD}. A semi-norm
\begin{equation}
|u|_L:= \|\boldsymbol{D}^s u\|_{L^2(\mathbb{R})} ~~\text{for}~~\widetilde{W}^{s}_L(\mathbb{R}) ~~\text{and}~~ |u|_R:= \|\boldsymbol{D}^{s*} u\|_{L^2(\mathbb{R})}
~~\text{for}~~\widetilde{W}^{s}_R(\mathbb{R}),
\end{equation}
is given with the corresponding norm

\begin{equation}
\quad \|u\|_{\star}:=(\|u\|^2_{L^2(\mathbb{R})}+|u|_\star^2)^{1/2},~~\star=L,R.
\end{equation}
 \end{definition}
\begin{remark}\label{rem:Notations}
Notice the special case   $\widetilde{W}^{0}_L(\mathbb{R})=\widetilde{W}^{0}_R(\mathbb{R})
=\widehat{H}^0(\mathbb{R})=L^2(\mathbb{R})$. 
\end{remark}
\begin{property}[Uniqueness of Weak Fractional R-L Derivatives \cite{V.Ging&Y.Li}]\label{uniquenessofR-L}
If $v\in L^1_{loc}(\mathbb{R})$ has a weak $s$-order left (or right) fractional derivative, then it is unique up to a set of zero measure.
\end{property}
Now we have the following characterization of  Sobolev space $\widehat{H}^s(\mathbb{R})$.
\begin{theorem}[\cite{V.Ging&Y.Li}]\label{thm:EquivalenceOfSpaces}
Given $s\geq 0$, $\widetilde{W}^{s}_L(\mathbb{R})$, $\widetilde{W}^{s}_R(\mathbb{R})$ and $\widehat{H}^s(\mathbb{R})$ are identical spaces with equal norms and semi-norms.
\end{theorem}
\begin{corollary}[\cite{V.Ging&Y.Li}]\label{cor:DensityOfC}
$u\in \widehat{H}^{s}(\mathbb{R})$ if and only if there exits a  sequence $\{u_n\}\subset C_0^\infty(\mathbb{R})$ such that $\{u_n\}, \{\boldsymbol{D}^su_n\}$ are Cauchy sequences in $L^2(\mathbb{R})$, with $\lim_{n\rightarrow \infty}u_n=u$. As a consequence, we have  $\lim_{n\rightarrow \infty}\boldsymbol{D}^{s}u_n=\boldsymbol{D}^{s}u$.
\\
Likewise,
\\
$u\in \widehat{H}^{s}(\mathbb{R}) $ if and only if there exits a  sequence $\{u_n\}\subset C_0^\infty(\mathbb{R})$ such that $\{u_n\}, \{\boldsymbol{D}^{s*}u_n\}$ are Cauchy sequences in $L^2(\mathbb{R})$, with $\lim_{n\rightarrow \infty}u_n=u$. As a consequence, we have $\lim_{n\rightarrow\infty}\boldsymbol{D}^{s*}u_n=\boldsymbol{D}^{s*}u$.
\end{corollary}
\section{Main Results} \label{section3}
Let us keep in mind that throughout the rest of paper, fractional derivatives are always understood in the weak sense defined in Section~\ref{section2}. The  main result in this work is the following.
\begin{theorem}\label{MAINTHEOREM}
Given $s\geq 0$. Let $f(\xi)$ be a function of form $\prod_{i=1}^n (1+\chi(s_i) \cdot(\pm2\pi  \xi i)^{s_i})$, where 
$0\leq s_i$ $ (i=1,\cdots, n)$, $\Sigma_{i=1}^n s_i\leq s$ and $\chi(s) =
\begin{cases}
1, \text{ if } \frac{s}{2} \in (0,\frac{1}{2}] \text{ or } \big[\frac{3}{2}+2k, \frac{5}{2}+2k \big], k \in \mathbb{N},\\
-1, \text{ if } \frac{s}{2}\in \big( \frac{1}{2}+2k, \frac{3}{2}+2k \big), ~ k \in \mathbb{N}.
\end{cases}$
Denote $\tau=s-\Sigma_{i=1}^n s_i$.

 Then the following is true:

 $v(x)\in \widehat{H}^s(\mathbb{R})$ if and only if there exists a $ u(x) \in \widehat{H}^\tau(\mathbb{R})$  such that $f(\xi)\cdot\widehat{v}(\xi)=\widehat{u}(\xi)$.
\end{theorem}
\begin{remark}
As convention, the complex power functions are understood as $(\mp i\xi)^{\sigma}=|\xi|^\sigma e^{\mp  \sigma \pi i \cdot \text{sign} (\xi)/2}$.
\end{remark}
Before embarking on the rigorous proof of the theorem above, let us first establish several lemmas which are of crucial importance later.

\begin{lemma}\label{Lemma1-new}
Fix $s\geq 0$. The following sets  are dense in $L^{2}(\mathbb{R})$ respectively:
\begin{itemize}
\item $M_1=\{w: w=\psi+\boldsymbol{D}^{s}\psi, ~\forall \psi\in C_0^\infty(\mathbb{R})\}$
\item $M_2=\{w: w=\psi-\boldsymbol{D}^{s}\psi, ~\forall \psi\in C_0^\infty(\mathbb{R})\}$
\item $M_3=\{w: w=\psi+\boldsymbol{D}^{s*}\psi, ~\forall \psi\in C_0^\infty(\mathbb{R})\}$
\item $M_4=\{w: w=\psi-\boldsymbol{D}^{s*}\psi, ~\forall \psi\in C_0^\infty(\mathbb{R})\}$
\end{itemize}
\end{lemma}
\begin{proof}
1. The proof is provided for $M_1$ only and the results for other sets follow analogously without essential differences.

2. To use Theorem~\ref{Thm:dense}, first notice $L^2(\mathbb{R})$ is a Hilbert space and $M_1\subset L^2(\mathbb{R})$ by Property~\ref{prop:Boundedness}. Furthermore, it is effortless to verify $M_1$ is a subspace of $L^2(\mathbb{R})$. Thus all the hypothesis of Theorem~\ref{Thm:dense} is met. 

3. Let us assume that $g\in L^2(\mathbb{R})$ and $(g, w)=0$ for any $w\in M_1$. The proof is done if this implies $g=0$ a.e..

4. Pick a non-zero function $\psi\in C^\infty_0(\mathbb{R})$, then by Plancherel Theorem we know $\widehat{\psi}(\xi)$ is a non-zero function, namely, $\widehat{\psi}(\xi)\neq 0$. On account of continuity of $\widehat{\psi}(\xi)$, there exists a non-empty open interval $(a,b)\subset \mathbb{R}$ such that $\widehat{\psi}(\xi)\neq 0$ on $(a,b)$.

5. Let $v(x)=\psi(\epsilon x)$, where $\epsilon$ is any positive fixed number. It's clear that $v(x)\in C^\infty_0(\mathbb{R})$ as well. Computing the Fourier transform of $w(x)=v(x)+\boldsymbol{D}^sv$ by Property~\ref{lem:FTFD} gives
\begin{equation}
\widehat{w}(\xi)=(1+(2\pi i \xi)^s)\widehat{v}(\xi)=(1+(2\pi i \xi)^s)\cdot\dfrac{1}{\epsilon}\widehat{\psi}(\dfrac{\xi}{\epsilon}).
\end{equation}
Notice $|1+(2\pi i \xi)^s|\neq 0$ a.e. and $\widehat{\psi}(\dfrac{\xi}{\epsilon})\neq 0$ on $(a\epsilon, b\epsilon)$. It follows that $\widehat{w}(\xi)\neq 0$ a.e. on $(a\epsilon, b\epsilon)$.

6. Set new function $G(-y) :=\int_\mathbb{R} g(x)w(x-y)\, {\rm d}x = \int_\mathbb{R} g(x)\, \tau_y w(x) \, {\rm d}x$. Using Property~\ref{pro:TranslationDerivative} gives $\tau_y w(x)\in M_1$ and therefore $G(-y)=0$ for any $y\in \mathbb{R}$ by our assumption in Step 3. And thus by Plancherel Theorem we know $\widehat{G}(\xi)= 0$ a.e. as well. 

7. Note that $G(y)=g(-x)*w(x)$ and $g,w\in L^2(\mathbb{R})$. Using convolution theorem (\cite{MR2218073}, Theorem 1.2, p.12) gives $\widehat{G}=\widehat{g(-x)}\cdot\widehat{w(x)}=0$, which implies $\widehat{g(-x)}(\xi)=0$ a.e.  on $(\epsilon a,\epsilon b)$ by recalling $\widehat{w}(\xi)\neq 0$ a.e. on $(a\epsilon, b\epsilon)$.

8. Because of the arbitrariness of $\epsilon$, we deduce $\widehat{g(-x)}(\xi)=0$ a.e. in $\mathbb{R}$ and therefore $g(x)=0$ a.e in $\mathbb{R}$ by another use of inverse Plancherel Theorem. This completes the whole proof.
\end{proof}
\begin{lemma}\label{normestimate-New}
 Given $0\leq s$,  $\psi \in C^\infty_0(\mathbb{R})$, then
 \begin{align}
  \|\psi\pm\boldsymbol{D}^{s}\psi\|^2_{L^2(\mathbb{R})} &=\|\psi\|^2_{L^2(\mathbb{R})}\pm2\cos(\frac{s}{2}\pi)\|\boldsymbol{D}^{s/2}\psi\|^2_{L^2(\mathbb{R})}+\|\boldsymbol{D}^{s}\psi\|^2_{L^2(\mathbb{R})}, \label{equation1-new}\\
 \|\psi\pm\boldsymbol{D}^{s*}\psi\|^2_{L^2(\mathbb{R})} &=\|\psi\|^2_{L^2(\mathbb{R})}\pm2\cos(\frac{s}{2}\pi)\|\boldsymbol{D}^{s/2}\psi\|^2_{L^2(\mathbb{R})}+\|\boldsymbol{D}^{s*}\psi\|^2_{L^2(\mathbb{R})}.\label{euqation2-new}
 \end{align}

 \end{lemma}
 \begin{proof}
1. The proof is established for one case of identities above only, namely $ \|\psi+\boldsymbol{D}^{s}\psi\|^2_{L^2(\mathbb{R})}$, and the others can be shown analogously by repeating the same procedure (even though involve left derivative and right derivative in different cases).

2. Since $\|\psi+\boldsymbol{D}^{s}\psi\|^2_{L^2(\mathbb{R})} = \|\psi\|^2_{L^2(\mathbb{R})}+\|\boldsymbol{D}^{s}\psi\|^2_{L^2(\mathbb{R})}+2 (\psi, \boldsymbol{D}^s \psi)$, we only need to show $(\psi, \boldsymbol{D}^s \psi)=\cos(\frac{s}{2}\pi)\|\boldsymbol{D}^{s/2}\psi\|^2_{L^2(\mathbb{R})}$.

3. Note that if we could show $(\psi, \boldsymbol{D}^s \psi)=(\boldsymbol{D}^{s/2*}\psi, \boldsymbol{D}^{s/2}\psi)$, then immediately $(\psi, \boldsymbol{D}^s \psi)=\cos(\frac{s}{2}\pi)\|\boldsymbol{D}^{s/2}\psi\|^2_{L^2(\mathbb{R})}$ follows by application of the second identity of Theorem 4.1 in \cite{V.Ging&Y.Li} and so the proof is done. 

4. The fact $(\psi, \boldsymbol{D}^s \psi)=(\boldsymbol{D}^{s/2*}\psi, \boldsymbol{D}^{s/2}\psi)$ could be verified by a straightforward calculation as follows. Let us rewrite $s=n-\delta$ with $0\leq \delta<1$, where $n$ is a positive integer. Notice $\psi\in C^\infty_0(\mathbb{R})$, if integer $n$ is even, using definition of R-L derivative and Lemma 2.2 (\cite{MR2218073}, p.73) gives
\begin{equation}\label{RepeatingEquation}
 (\psi, \boldsymbol{D}^s \psi)=(\psi, \frac{\rm d}{\rm dx^{(n/2)}}\boldsymbol{D}^{n/2-\delta} \psi)=(\psi, \frac{\rm d}{\rm dx^{(n/2)}}\boldsymbol{D}^{-\delta} \psi^{(n/2)}).
\end{equation}
 Using integration by parts and semigroup property \ref{semigroup-new} the last term becomes
 $$ (\psi, \frac{\rm d}{\rm dx^{(n/2)}}\boldsymbol{D}^{-\delta} \psi^{(n/2)})=(\boldsymbol{D}^{n/2*}\psi, \boldsymbol{D}^{-\delta/2}\boldsymbol{D}^{-\delta/2}\psi^{(n/2)}).$$
To simplify the right-hand side, notice the fact that $\boldsymbol{D}^{-\delta/2}\psi^{(n/2)}=\boldsymbol{D}^{n/2-\delta/2}\psi$ by applying Lemma 2.2 (\cite{MR2218073}, p.73), and this gives $\boldsymbol{D}^{-\delta/2}\psi^{(n/2)}$ $\in L^p(\mathbb{R})$ for $p\geq 1$ by Property~\ref{prop:Boundedness}. Thus
 $$ (\boldsymbol{D}^{n/2*}\psi, \boldsymbol{D}^{-\delta/2}\boldsymbol{D}^{-\delta/2}\psi^{(n/2)})=(\boldsymbol{D}^{s/2*}\psi, \boldsymbol{D}^{s/2}\psi)$$
 follows immediately from Property~\ref{prp:IntegrationExchange} and another use of Lemma 2.2 (\cite{MR2218073}, p.73). Therefore $(\psi, \boldsymbol{D}^s \psi)=(\boldsymbol{D}^{s/2*}\psi, \boldsymbol{D}^{s/2}\psi)$.

5. If integer $n$ is odd, using definition of R-L derivative and Lemma 2.2 (\cite{MR2218073}, p.73), it is not difficult to verify
 $$(\psi, \boldsymbol{D}^s \psi)=(\psi, \frac{\rm d}{\rm dx^{((n+1)/2)}}\boldsymbol{D}^{-(\delta+1)} \psi^{((n+1)/2)}).$$ Now $n+1$ is even. Repeating above procedure starting from Equation~\eqref{RepeatingEquation} gives us the desired result, which is omitted here. Thus we are done.
 \end{proof}
\begin{lemma}\label{Existence-new}
(a). Let $\chi(s) =
\begin{cases}
1, \text{ if } \frac{s}{2} \in (0,\frac{1}{2}] \text{ or }  \big[\frac{3}{2}+2k, \frac{5}{2}+2k \big], k \in \mathbb{N},\\
-1, \text{ if } \frac{s}{2}\in \big( \frac{1}{2}+2k, \frac{3}{2}+2k \big), ~ k \in \mathbb{N}. 
\end{cases}$

Given $0\leq s$, there exists a one to one and onto map $T: v\mapsto u$ from $\widehat{H}^s(\mathbb{R})$ to $L^2(\mathbb{R})$ such that
\begin{equation}\label{MainResult-equation1}
v+\chi(s)\cdot\boldsymbol{D}^sv=u.
\end{equation}
Analogously, there exists a one to one and onto map $T^*: v\mapsto u$ from $\widehat{H}^s(\mathbb{R})$ to $L^2(\mathbb{R})$ such that
\begin{equation}\label{MainResult-equation2}
v+\chi(s)\cdot\boldsymbol{D}^{s*}v=u.
\end{equation}
(b). Furthermore, for each case, $v\in \widehat{H}^{s+t}(\mathbb{R})$ if and only if $u\in \widehat{H}^{t}(\mathbb{R})$ for any $t>0.$
\end{lemma}
\begin{proof}
1. For part (a), the proof is shown only for Equality~\eqref{MainResult-equation1} since the second one could be shown similarly.

2. First we show the map is onto. Fix a $u\in L^2(\mathbb{R})$, and without loss of generality, we assume $\chi(s)=1$. Invoking Lemma~\ref{Lemma1-new}, there exists a Cauchy sequence $\{v_n-\boldsymbol{D}^sv_n\}$ converging to $u$ in $L^2(\mathbb{R})$, where $\{v_n\}\subset C_0^\infty(\mathbb{R})$. This implies that  $$\|(v_m-v_n)-(\boldsymbol{D}^sv_m-\boldsymbol{D}^sv_n)\|^2_{L^2(\mathbb{R})} \rightarrow 0 \quad\text{ as} \quad m,n \rightarrow \infty.$$
We next employ Lemma~\ref{normestimate-New} to equivalently obtain that, as $m,n \rightarrow \infty$, $$\|v_m-v_n\|^2_{L^2(\mathbb{R})}+2\cos(\frac{s}{2}\pi)\|\boldsymbol{D}^{s/2}v_m-\boldsymbol{D}^{s/2}v_n\|^2_{L^2(\mathbb{R})}+\|\boldsymbol{D}^{s}v_m-\boldsymbol{D}^{s}v_n\|^2_{L^2(\mathbb{R})}$$
$\rightarrow 0.$
Now notice that $2\cos(\frac{s}{2}\pi)\geq 0$ since $\chi(s)=1$. Therefore, we deduce that $\{v_n\}$ and $\{\boldsymbol{D}^sv_n\}$ are Cauchy sequences respectively. This concludes that , if we denote $v=\lim_{n\rightarrow \infty}v_n$, $v+\chi(s)\cdot\boldsymbol{D}^s v=u$ and $v\in \widehat{H}^s(\mathbb{R})$ by Corollary~\ref{cor:DensityOfC} . 

3. We now show the map is one to one. First notice the fact that $\boldsymbol{D}^sv \in L^2(\mathbb{R})$ if $v\in \widehat{H}^s(\mathbb{R})$ by the definition of weak fractional derivative, therefore $v+\chi(s)\cdot\boldsymbol{D}^sv\in L^2(\mathbb{R})$, which means $T: v\mapsto u$ is a map by the uniqueness of weak fractional derivative~\ref{uniquenessofR-L}. Second, assume there exists another $\tilde{v}$ such that $\tilde{v}+\boldsymbol{D}^s \tilde{v}=v+\boldsymbol{D}^s v$, then $(\tilde{v}-v)+\boldsymbol{D}^s( \tilde{v}- v)=0$ immediately follows. By using a standard norm estimate argument and application of Lemma~\ref{normestimate-New}, it is easy to obtain $\|\tilde{v}-v\|^2_{L^2(\mathbb{R})}=0$, which implies $\tilde{v}=v$. Thus $T: v\mapsto u$ is a one to one map. 

4. We remain to show part (b). First suppose  $v\in \widehat{H}^{s+t}(\mathbb{R})$, obviously  $v\in \widehat{H}^{t}(\mathbb{R})$ since $\widehat{H}^{s+t}(\mathbb{R})\subset \widehat{H}^{t}(\mathbb{R})$. 
 Now we claim $\boldsymbol{D}^{s}v\in \widehat{H}^t(\mathbb{R})$. To see this, using Corollary~\ref{cor:DensityOfC}, there exist Cauchy sequences $\{v_n\}\subset C_0^\infty(\mathbb{R})$  and $\{\boldsymbol{D}^{s+t}v_n\}$ in $L^2(\mathbb{R})$ such that $v_n\rightarrow v, \boldsymbol{D}^{s+t}v_n\rightarrow \boldsymbol{D}^{s+t}v$ as $n\rightarrow \infty$. 
 Thus for any $\psi\in C_0^\infty(\mathbb{R})$, we have  $$(\boldsymbol{D}^{s}v, \boldsymbol{D}^{t*} \psi)=\lim_{m\rightarrow \infty}(\boldsymbol{D}^{s}v_n, \boldsymbol{D}^{t*}\psi)=\lim_{m\rightarrow \infty}(\boldsymbol{D}^{s+t}v_n, \psi)=(\boldsymbol{D}^{s+t}v,\psi).$$
By definition of weak fractional derivative we conclude $\boldsymbol{D}^{s}v\in \widehat{H}^t(\mathbb{R})$ and thus $u(x)\in \widehat{H}^{t}(\mathbb{R})$.

5. The last step is to show $u(x)\in \widehat{H}^{t}(\mathbb{R})$ implies $v\in \widehat{H}^{s+t}(\mathbb{R})$. First notice that it is always possible to rewrite $t=t_1+\cdots+t_n$ such that $t_i\geq 0,$  and $t_i\leq s$ $(i=1,\cdots, n)$. Thus, $ \widehat{H}^s(\mathbb{R})\subset \widehat{H}^{t_1}$ and $ \widehat{H}^t(\mathbb{R})\subset\widehat{H}^{t_1}$, and this deduces $\boldsymbol{D}^sv=u-v\in \widehat{H}^{t_1}(\mathbb{R})$. According to Theorem~\ref{thm:EquivalenceOfSpaces}, $\boldsymbol{D}^sv\in \widetilde{W}^{t_1}_L(\mathbb{R}) $, which by definition means that there exists a $Q(x)\in L^2(\mathbb{R})$ such that 
\begin{equation} \label{Middle1Equation}
(\boldsymbol{D}^sv, \boldsymbol{D}^{t_1*}\psi) = (Q, \psi) ~\forall \psi \in C_0^\infty(\mathbb{R}).
\end{equation}
Again, by invoking Corollary~\ref{cor:DensityOfC}, the left-hand side of Equation~\eqref{Middle1Equation} becomes
\begin{equation}
\begin{aligned}
(\boldsymbol{D}^{s}v, \boldsymbol{D}^{t_1*} \psi)&=\lim_{m\rightarrow \infty}(\boldsymbol{D}^{s}v_n, \boldsymbol{D}^{t_1*}\psi)\\
&=\lim_{m\rightarrow \infty}(v_n, \boldsymbol{D}^{(s+t_1)*}\psi)\\
&=(v,\boldsymbol{D}^{(s+t_1)*}\psi).
\end{aligned}
\end{equation}
This implies that $v\in \widehat{H}^{s+t_1}$ by another utilization of definition of $\widetilde{W}^{s}_L(\mathbb{R}) $ and  Theorem~\ref{thm:EquivalenceOfSpaces}. By repeating this procedure we could increase the regularity of $v$ gradually, namely, we could show that $v\in \widehat{H}^{s+t_1+t_2}(\mathbb{R}), \cdots , v\in \widehat{H}^{s+t_1+\cdots +t_n}(\mathbb{R})=\widehat{H}^{s+t}(\mathbb{R})$, as desired. Thus the whole proof is complete.
\end{proof}
\begin{lemma}\label{LastLemma}
(a). Let $\chi(s) =
\begin{cases}
1, \text{ if } \frac{s}{2} \in (0,\frac{1}{2}] \text{ or }  \big[\frac{3}{2}+2k, \frac{5}{2}+2k \big], k \in \mathbb{N},\\
-1, \text{ if } \frac{s}{2}\in \big( \frac{1}{2}+2k, \frac{3}{2}+2k \big), ~ k \in \mathbb{N}. 
\end{cases}$

Given $0\leq s$, there exists a one to one and onto map $T: v\mapsto u$ from $\widehat{H}^s(\mathbb{R})$ to $L^2(\mathbb{R})$ such that
\begin{equation}\label{Fourier-equation1}
(1+\chi(s)\cdot(2\pi \xi i)^s)\cdot\widehat{v}(\xi)=\widehat{u}(\xi).
\end{equation}
Analogously, there exists a one to one and onto map $T^*: v\mapsto u$ from $\widehat{H}^s(\mathbb{R})$ to $L^2(\mathbb{R})$ such that
\begin{equation}\label{Fourier-equation2}
(1+\chi(s)\cdot(-2\pi \xi i)^s)\cdot\widehat{v}(\xi)=\widehat{u}(\xi).
\end{equation}
(b). Furthermore, for each case, $v\in \widehat{H}^{s+t}(\mathbb{R})$ if and only if $u\in \widehat{H}^{t}(\mathbb{R})$ for any $t>0.$
\end{lemma}
\begin{proof}
1. For part (a), the proof is shown only for first Equation~\eqref{Fourier-equation1}, the other one follows by repeating the same procedure.

2. From Lemma~\ref{Existence-new}, there exists a one to one and onto map $T: v\mapsto u$ from $\widehat{H}^s(\mathbb{R})$ to $L^2(\mathbb{R})$ such that
\begin{equation}
v+\chi(s)\cdot\boldsymbol{D}^sv=u.
\end{equation}
Taking the Plancherel transform at both sides gives
$$\widehat{v}+\chi(s)\cdot\widehat{\boldsymbol{D}^sv}=\widehat{u}. $$
So far, we could not yet directly apply Property~\ref{lem:FTFD} to obtain $\widehat{\boldsymbol{D}^sv}(\xi)=(2\pi \xi i)^s\cdot\widehat{v}(\xi)$ since the condition is not met, namely $v$ does not necessarily belong to $C_0^\infty(\mathbb{R})$. The argument for $\widehat{\boldsymbol{D}^sv}(\xi)=(2\pi \xi i)^s\cdot\widehat{v}(\xi)$ is justified in the proof of Theorem 3.3 (\cite{V.Ging&Y.Li}) and will not be repeated here. Thus,
$$(1+\chi(s)\cdot(2\pi \xi i)^s)\cdot\widehat{v}(\xi)=\widehat{u}(\xi).$$

3. The proof of part (b) directly follows from the part (b) of Lemma~\ref{Existence-new}, which completes the whole proof.
\end{proof}
Now we are in the position to prove Theorem~\ref{MAINTHEOREM}.
\begin{proof}
1. We intend to construct a one to one and onto map $v\mapsto u$ from $\widehat{H}^s(\mathbb{R})$ to $\widehat{H}^{\tau}(\mathbb{R})$ such that
$$f(\xi)\cdot \widehat{v}(\xi)=\widehat{u}(\xi).$$ Then the proof is done. 

2. Utilizing both part (a) and (b) of Lemma~\ref{LastLemma}, we know there exists a one to one and onto map $T_1: v\mapsto u_1$ from $\widehat{H}^{s}(\mathbb{R})$ to $\widehat{H}^{s-s_1}(\mathbb{R})$ such that
\begin{equation}
(1+\chi(s)\cdot(2\pi \xi i)^{s_1})\cdot\widehat{v}(\xi)=\widehat{u_1}(\xi).
\end{equation}
Repeating the same application of Lemma~\ref{LastLemma} for $u_1$, we know 
there exists a one to one and onto map $T_2: u_1\mapsto u_2$ from $\widehat{H}^{s-s_1}(\mathbb{R})$ to $\widehat{H}^{s-s_1-s_2}(\mathbb{R})$ such that
\begin{equation}
(1+\chi(s)\cdot(2\pi \xi i)^{s_2})\cdot\widehat{u_1}(\xi)=\widehat{u_2}(\xi).
\end{equation}
By repeating the same procedure, we obtain one to one and onto maps $T_2,\cdots, T_n$, where $T_n:u_{n-1}\rightarrow u_n$ from $\widehat{H}^{s-s_1-\cdots-s_{n-1}}(\mathbb{R})$ to $\widehat{H}^{s-s_1-\cdots-s_n}(\mathbb{R})$ such that
\begin{equation}
(1+\chi(s)\cdot(2\pi \xi i)^{s_n})\cdot\widehat{u_{n-1}}(\xi)=\widehat{u_n}(\xi).
\end{equation}
Recall $\tau=s-\Sigma_{i=1}^n s_i$, therefore, $T=T_1\circ T_2\cdots \circ T_n:v \mapsto u_n$ is a one to one and onto map $v\mapsto u_n $ from $\widehat{H}^s(\mathbb{R})$ to $\widehat{H}^{\tau}(\mathbb{R})$, satisfying $f(\xi)\cdot\widehat{v}(\xi)=\widehat{u_n}(\xi)$.  This completes Step 1 above by regarding $u_n$ as $u$, and thus completes the whole proof for Theorem~\ref{MAINTHEOREM}.
\end{proof}

\begin{appendices}
\renewcommand{\theequation}{\thesection.\arabic{equation}}
\numberwithin{equation}{section}
\section{Several Definitions and Theorems}

\begin{definition}[Sobolev Spaces Via Fourier Transform]\label{thm:FTHsR}
Let $\mu\geq 0$. Define
\begin{equation}
\widehat{H}^\mu(\mathbb{R}) = \left \{ w \in L^2(\mathbb{R}) : \int_{\mathbb{R}} (1 + |2\pi\xi|^{2\mu}) |\widehat{w}(\xi) |^2 \, {\rm d} \xi < \infty \right \},
\end{equation}
where $\widehat{w}$ is Plancherel  transform defined in Theorem~\ref{thm:PAR}.
The space is endowed with semi-morn 
\begin{equation}
|u|_{\widehat{H}^\mu(\mathbb{R})}:=\||2\pi\xi|^\mu \widehat{u}\|_{L^2(\mathbb{R})},
\end{equation}
and norm
\begin{equation}
\|u\|_{\widehat{H}^\mu (\mathbb{R})}:=\left(\|u\|^2_{L^2(\mathbb{R})} +|u|^2_{\widehat{H}^\mu(\mathbb{R})}\right)^{1/2}.
\end{equation}
\end{definition}
\noindent
And it is well-known that $\widehat{H}^\mu(\mathbb{R})$ is a Hilbert space. 
\begin{definition}[Fourier Transform] \label{def:FT}
Given a function $f: \mathbb{R} \rightarrow \mathbb{R}$, the Fourier Transform of $f$ is defined as
$$
\mathcal{F}(f)(\xi):=\int_{-\infty}^{\infty}e^{-2\pi i x\xi}f(x)~ {\rm d}x \quad \forall \xi \in \mathbb{R}.
$$
\end{definition}
\begin{theorem}[Plancherel Theorem ( \cite{MR924157} p. 187)] \label{thm:PAR}

One can associate to each $f\in L^2(\mathbb{R})$ a function $\widehat{f}\in L^2(\mathbb{R})$ so that the following properties hold:
 \begin{itemize}
 \item If $f\in L^1(\mathbb{R})\cap L^2(\mathbb{R})$, then $\widehat{f} $ is the defined Fourier transform of $f$ in Definition \ref{def:FT}.
 \item For every $f\in L^2(\mathbb{R})$, $\|f\|_2=\|\widehat{f}\|_2$.
 \item The mapping $f\rightarrow \widehat{f}$ is a Hilbert space isomorphism of $L^2(\mathbb{R})$ onto $L^2(\mathbb{R})$.
 \end{itemize}
\end{theorem}
\begin{theorem}(\cite{MR2597943}, p. 189 )\label{thm:ParsevalFormula}
Assume $u,v\in L^2(\mathbb{R}^n)$. Then
\begin{itemize}
\item $\displaystyle \int_{\mathbb{R}^n} u\overline{v}=\int_{\mathbb{R}^n} \widehat{u}\overline{\widehat{v}}.$
\item $u=(\widehat{u})^{\vee}.$
\end{itemize}
\end{theorem}
\begin{theorem}[\cite{MR2328004}, Theorem 4.3-2, p. 191]\label{Thm:dense}
Let $(X,(\cdot,\cdot))$ be a Hilbert space and let $Y$ be a subspace of $X$, then $\overline{Y}=X$ if and only if element $x\in X$ that satisfy $(x,y)=0$ for all $y\in Y$ is $x=0$.
\end{theorem}
\end{appendices}


\end{document}